\documentclass[11pt]{article}
\usepackage[letterpaper,hmargin=1in,vmargin=1in]{geometry}
\usepackage{graphicx}
\usepackage{amsmath,amssymb,amsthm,mathtools}
\usepackage{url}
\usepackage{color}
\usepackage{cite}
\usepackage{xparse}
\usepackage{hyperref}
\hypersetup{colorlinks=true,linkcolor=[rgb]{0,0,0},citecolor=[rgb]{0,0,0},urlcolor=[rgb]{0,0,0}}
\usepackage{thmtools}
\usepackage{thm-restate}
\theoremstyle{plain}
\newtheorem{theorem}{Theorem}
\newtheorem{lemma}{Lemma}[section]

\newtheorem{corollary}[lemma]{Corollary}
\newtheorem{problem}[lemma]{Problem}

\theoremstyle{definition}
\newtheorem{definition}[theorem]{Definition}

\theoremstyle{remark}

\newcommand{\floor}[1]{\left\lfloor #1\right\rfloor}

\newcommand{\ang}[1]{\left\langle #1 \right\rangle}
\newcommand{\abs}[1]{\left\lvert #1 \right\rvert}
\newcommand{\RE}{\mathbb{R}}

\newcommand{\NN}{\mathbb{N}}

\newcommand{\ST}{\,:\,}
\newcommand{\inner}[2]{\ang{#1,#2}}

\newcommand{\etal}{\textit{et al.}}
\newcommand{\SP}{\kern+1pt}
\newcommand{\polar}{\circ}
\DeclareMathOperator{\Aff}{Aff}
\DeclareMathOperator{\conv}{conv}

\DeclareMathOperator{\interior}{int}
\DeclareMathOperator{\dist}{dist}
\NewDocumentCommand{\distHilb}{O{}}{\dist^{H}_{#1}}
\NewDocumentCommand{\distFunk}{O{}}{\dist^{F}_{#1}}
\newcommand{\Mconv}{\widehat{M}}   % convexified packing
\newcommand{\Mord}{M}              % ordinary packing
\title{Duality Bounds for Convexified Packing in Hilbert Geometry}

\author{%
	Sunil Arya\\
		Department of Computer Science and Engineering \\
		Hong Kong University of Science and Technology \\
		Clear Water Bay, Kowloon, Hong Kong\\
		arya@cse.ust.hk
		\and
	David M. Mount\\
		Department of Computer Science and \\
		Institute for Advanced Computer Studies \\
		University of Maryland \\
		College Park, Maryland 20742 \\
		mount@umd.edu
}

\date{}

\begin{document}

\maketitle

%-----------------------------------------------------------------------
\begin{abstract}
Let $G$ and $K$ be convex bodies in $\RE^d$, where $0 \in \interior G$ and $G \subset \interior K$. Given $\alpha > 0$, the Hilbert packing number $\Mord_H(G, K; \alpha)$ is the maximum cardinality of a set of points in $G$, each pair of which is separated by a distance of at least $\alpha$ in the Hilbert geometry defined by $K$. The Hilbert convexified packing number $\Mconv_H(G, K; \alpha)$ is the maximum length of a sequence of points in $G$, such that each point is separated by distance at least $\alpha$ from the convex hull of its predecessors. We prove a dimension-free primal--polar bound for convexified packing in Hilbert geometry. Letting $G^\polar$ and $K^\polar$ denote the polar bodies, we show that there exist absolute constants $C, c > 0$ such that, for every $\alpha > 0$,
\[
    \Mconv_H(G, K; \alpha) 
        ~ \leq ~ C \cdot \Mconv_H(K^\polar, G^\polar; c\alpha)^2 \Mord_H(K^\polar, G^\polar; c\alpha).
\]
As a direct corollary, we have
\[
    \Mconv_H(G, K; \alpha) 
        ~ \leq ~ C \cdot \Mord_H(K^\polar, G^\polar; c\alpha)^3.
\]
Thus, the primal convexified packing number is bounded by a fixed polynomial in the ordinary packing number for the reversed polar bodies, with absolute constants independent of the dimension. This is motivated by the duality conjecture for packing and covering numbers, which relates covering $G$ by $K$ to covering $K^\polar$ by $G^\polar$. Our results represent a key step in extending the work of Artstein, Milman, Szarek, and Tomczak-Jaegermann from normed spaces to Hilbert geometries.
\end{abstract}
%-----------------------------------------------------------------------

%=======================================================================
\section{Introduction} \label{sec:introduction}
%=======================================================================

Packing and covering numbers are fundamental measures of metric complexity. Given convex bodies $G$ and $K$ in $\RE^d$, the \emph{packing number} of $G$ by $K$, denoted $M(G,K)$, is the maximal number of disjoint translates of $K$ by elements of $G$, and the \emph{covering number}, $N(G,K)$, is the minimal number of translates of $K$ needed to cover $G$. In asymptotic convex geometry, the interaction with polarity gives rise to the duality theory of metric entropy and, equivalently, to comparisons between the entropy numbers of a compact linear operator and those of its adjoint. A central problem in this area is the duality conjecture for packing and covering numbers, inspired by the work of Pietsch \cite{Pie72} in the operator-theoretic context, which involves the relationship between covering numbers for $G$ by $K$ and covering numbers for the reversed polar bodies, $K^\polar$ by $G^\polar$. It is conjectured that there exist absolute constants $a,b>1$ such that, for any two centrally symmetric convex bodies $G$ and $K$ in any dimension,
\[
    b^{-1} \log N(K^\polar,a G^\polar)
        ~ \leq ~ \log N(G, K)
        ~ \leq ~ b \log N(K^\polar, a^{-1} G^\polar).
\]

For centrally symmetric convex bodies, an early (dimension-sensitive) polarity estimate was obtained by K\"onig and Milman~\cite{KoM87}. The search for a dimension-free form of entropy duality led to work of Bourgain, Pajor, Szarek, Tomczak-Jaegermann, and others~\cite{BPS89}. Artstein, Milman, and Szarek proved the conjectured duality when one of the bodies is an ellipsoid~\cite{AMS04a}. The general theory was further developed by Artstein, Milman, Szarek, and Tomczak-Jaegermann through the introduction of the useful concept of \emph{convexified packings}~\cite{AMS04b} (defined in Section~\ref{sec:preliminaries}), denoted by $\Mconv$.

As emphasized by Artstein {\etal}, convexified packing provides a useful way to divide the duality problem into two distinct stages:
\begin{enumerate}\setlength{\itemsep}{-0.5ex}\setlength{\parsep}{0pt}
\item[(i)] a \emph{duality stage}, relating convexified packing in the original space to packing quantities in the dual space and
\item[(ii)] a \emph{geometric stage}, comparing convexified and ordinary packing within each space separately.
\end{enumerate}
The first stage provides a bridge between the primal and dual geometries and is therefore where polarity enters. The second stage consists of separate geometric comparisons on the primal and dual sides. This division allows the duality stage to be studied independently of the comparisons needed to recover ordinary packing or covering numbers.

A key element in the analysis of convexified packings is separating a point from a convex set. In normed spaces, this is based on linear separation, in the sense that a single linear functional suffices to separate a point from a suitable neighborhood of a convex hull. The resulting functional witnesses can then be organized into a convexified packing in the polar body. In particular, for centrally symmetric convex bodies $K, G\subseteq\RE^d$, Artstein, Milman, Szarek, and Tomczak-Jaegermann \cite[Theorem~2]{AMS04b} proved the dimension-free bound
\[
    \widehat M(G,K)
        ~ \leq ~ \widehat M(K^\polar,G^\polar/2)^2.
\]
This establishes the essential polarity relation needed for the duality stage of their analysis.

The broad goal motivating this paper is to develop an analogous duality program in Hilbert geometry. This setting presents significant additional difficulties. A normed metric is translation invariant, and its unit balls are translates of one fixed convex body. By contrast, the shapes of metric balls in a Hilbert geometry vary based on their center points. Moreover, the Hilbert metric is intrinsically two-sided, since it is the symmetrization of two asymmetric Funk metrics. Consequently, Hilbert separation is not naturally represented by a single polar functional.

In this paper, we establish the first stage (the duality stage) of this program in Hilbert geometry. We develop a Hilbert-geometric replacement for the functional-witness argument used in normed spaces and prove a dimension-free primal--polar estimate for convexified packing. This provides the Hilbert-geometric counterpart of the convexified-packing polarity estimate that underlies the normed-space approach. Our result does not yet give a full duality theorem for ordinary packing or covering numbers. Completing the program would require the second stage (the geometric stage), that is, a comparison between ordinary and convexified packing within a single Hilbert geometry, with constants independent of the dimension. 

Every convex body $K$ in $\RE^d$ defines a \emph{Hilbert geometry} over its interior. This geometry is based on a metric, which arises as the symmetrization of the Funk weak metric on $K$. The Hilbert metric has several useful properties. It is projectively invariant and defines a Finsler structure. When $K$ is an ellipsoid, it coincides with the Beltrami--Klein model of hyperbolic geometry, and when $K$ is a simplex, it is isometric to a normed vector space. More generally, every finite-dimensional normed space arises as a rescaled local limit of some Hilbert geometry. (See, e.g., \cite{Bus55, Fai24, FoK05, PaT14a} for further information.)

The Hilbert-geometric duality problem is therefore not merely analogous to its normed-space counterpart. 
Since every finite-dimensional normed space arises as a rescaled local limit of a Hilbert geometry, a dimension-free duality theorem for ordinary Hilbert packing or covering, formulated uniformly over Hilbert geometries, would imply the corresponding normed-space theory through this local-limit correspondence. 
Thus, the Hilbert problem should be viewed as a genuine extension of the classical entropy-duality problem.

Throughout the paper, let $G$ and $K$ be convex bodies in $\RE^d$ satisfying
\[
    0 \in \interior G
    \qquad\text{and}\qquad
    G \subset \interior K.
\]
We regard $G$ as a compact subset of the Hilbert geometry defined by $K$. For $\alpha > 0$, define the (ordinary) \emph{Hilbert packing number} $\Mord_H(G, K; \alpha)$ to be the maximum cardinality of an $\alpha$-separated subset of $G$ in this geometry. The corresponding \emph{Hilbert convexified packing number}, denoted $\Mconv_H(G, K; \alpha)$, is the maximum length of a sequence of points in $G$ such that each point has Hilbert distance at least $\alpha$ from the convex hull of its predecessors. (See Section~\ref{sec:preliminaries} for full definitions.) Every convexified packing is an ordinary packing, and hence
\[
    \Mconv_H(G, K; \alpha)
        ~ \leq ~ \Mord_H(G, K; \alpha).
\]

Our focus is on the relationships between Hilbert packing numbers for $G$ in the geometry defined by $K$ and the corresponding packing numbers for the reversed polar pair. The assumptions on $G$ and $K$ imply $0 \in \interior K^\polar$ and $K^\polar \subset \interior G^\polar$. Thus $K^\polar$ is a compact subset of the Hilbert geometry defined by $G^\polar$. Our main result bounds the convexified packing complexity of $(G, K)$ by packing quantities for $(K^\polar, G^\polar)$.

%-----------------------------------------------------------------------
\begin{restatable}[Main theorem]{theorem}{thmMain}
\label{thm:main}
There exist absolute constants $C,c>0$ such that, for every pair of convex bodies $G$ and $K$ in $\RE^d$ satisfying $0 \in \interior G$ and $G \subset \interior K$, and every $\alpha > 0$,
\[
    \Mconv_H(G, K; \alpha)
        ~ \leq ~ C \cdot \Mconv_H(K^\polar, G^\polar; c \alpha)^2 \Mord_H(K^\polar, G^\polar; c \alpha).
\]
\end{restatable}
%-----------------------------------------------------------------------

The constants are independent of the dimension, and the proof gives $C = 11$ and $c = 3/16$. Since convexified packing is bounded above by ordinary packing, the theorem immediately yields
\[
    \Mconv_H(G, K; \alpha)
        ~ \leq ~ C \cdot \Mord_H(K^\polar, G^\polar; c \alpha)^3.
\]
Thus the primal convexified packing number is bounded by a fixed (dimension-independent) polynomial in the ordinary packing number of the reversed polar pair.

The mixed form of Theorem~\ref{thm:main} reflects a remaining difficulty in the Hilbert setting. Two of the factors on the polar side are convexified packing numbers, whereas the third is an ordinary packing number. The ordinary factor enters only in the final stage of the proof, where one coordinate of a sequence of polar witness pairs is controlled by an ordinary packing net. We do not know whether this factor can be replaced by another convexified packing factor.

Nevertheless, within the two-stage framework described above, Theorem~\ref{thm:main} supplies the principal duality step in the broader program. Its proof develops a substantially more elaborate Hilbert-geometric counterpart of the single-functional argument used in normed spaces, adapted to the intrinsically two-sided nature of Hilbert distance. Thus, within the convexified-packing approach, the theorem resolves the interaction between the primal and polar Hilbert geometries. For example, suppose that ordinary Hilbert packings could be bounded polynomially by Hilbert convexified packings, with dimension-free constants. 
Combining such an intrinsic comparison with Theorem~\ref{thm:main} and the inequality $\Mconv_H \leq \Mord_H$ would yield a dimension-free polynomial duality theorem for ordinary packing. The remaining comparison would concern only one Hilbert geometry at a time, rather than the interaction between a primal geometry and its polar. We discuss this implication more precisely in Section~\ref{sec:conclusion}.

The dimension-free character of Theorem~\ref{thm:main} contrasts with the authors' earlier polarity theorem for covering numbers in Hilbert geometry, whose constants grow exponentially with the dimension \cite{ArM26b}. Polarity also appears elsewhere in Hilbert and Funk geometries. For example, the Hilbert horofunction boundary is governed in part by the extreme sets of the polar body \cite{Wal08}, and Faifman established projective-polar duality for several geometric invariants in Funk geometry~\cite{Fai24}. The present result concerns a different aspect of polarity: the metric complexity of a compact set and that of its reversed polar pair.

%-----------------------------------------------------------------------
\subsection*{Overview of the Proof}
%-----------------------------------------------------------------------

The proof begins with a polar representation of the Hilbert metric. Each pair of polar witnesses determines a logarithmic pair-potential, which we denote by $\Phi$. The Hilbert distance is recovered by maximizing the change of such a potential between the two points.

Let $\ang{x_1, x_2, \dots}$ be a sequence of points in $G$ such that for all $j \geq 2$, $x_j$ is at Hilbert distance at least $\alpha$ from the convex hull of the preceding points. Pointwise separation provides a pair of polar witnesses for each point of the preceding hull, but this pair may depend on the point. After normalizing the affine functions $f_y$ at $x_j$, we apply a minimax argument to obtain a single pair $y_j, z_j \in K^\polar$ whose associated pair-potential $\Phi_j$ satisfies
\[
    \Phi_j(x_j)-\Phi_j(c)
        ~ \geq ~ 2 \alpha \qquad\text{for every $c \in \conv\{x_1, \dots, x_{j-1}\}$}.
\]
This finite-hull two-pole separation lemma is the Hilbert-geometric analog of separation by one linear functional.

The witness pair associated with $x_j$ has a relative normalization measured by the value $\Phi_j(x_j)$. We prove a polar oscillation estimate that bounds the total range of these values by a convexified packing number for $(K^\polar, G^\polar)$. This permits a decomposition of the witness pairs into a bounded number of scale classes. Within each class, the pairs form a special packing in $K^\polar \times K^\polar$, which we call a \emph{diagonally convexified packing}, in which both coordinates must be convexified using the same coefficients.

Finally, we prove a mixed product estimate for diagonal packings. One polar coordinate is partitioned using an ordinary Hilbert packing net. Within each part, the convexity of Hilbert balls forces the other coordinate to form a convexified packing. Quantitatively, we obtain
\[
    \Mconv_{\mathrm{diag},H}(K^\polar, G^\polar; \eta)
        ~ \leq ~ \Mord_H \!\left( K^\polar, G^\polar; \frac{\eta}{4} \right) \Mconv_H \!\left( K^\polar, G^\polar; \frac{\eta}{2} \right).
\]
Combining this estimate with the scale decomposition proves Theorem~\ref{thm:main}.

\bigskip

The remainder of the paper is organized as follows. Section~\ref{sec:preliminaries} introduces the Funk and Hilbert metrics, their polar-potential representations, and the packing quantities used throughout. Section~\ref{sec:two-pole-tools} develops normalized nonnegative affine functions, proves the finite-hull two-pole separation lemma, and establishes the polar oscillation estimate. Section~\ref{sec:diagonal-convexification} introduces diagonal convexified packing, proves the scale reduction and the mixed diagonal product estimate, and establishes Theorem~\ref{thm:main}. Section~\ref{sec:conclusion} discusses the remaining steps toward ordinary packing duality and formulates the fully convexified diagonal product problem.

%=======================================================================
\section{Preliminaries} \label{sec:preliminaries}
%=======================================================================

We present the notation and terminology used throughout the paper. We use $\inner{\cdot}{\cdot}$ to denote the standard inner product on $\RE^d$. Let $0$ denote the origin. A \emph{convex body} in $\RE^d$ is a compact convex set with nonempty interior. Given a convex set $K$, we denote its interior by $\interior K$. Throughout, $G$ and $K$ denote convex bodies in $\RE^d$ such that $0 \in \interior G$ and $G \subset \interior K$. 

Given a nonempty set $K \subset \RE^d$, its \emph{polar}, denoted by $K^\polar$, is defined by 
\[
	K^\polar
		~ = ~ \bigl\{ y \in \RE^d \ST \inner{x}{y} \leq 1 \text{ for all } x \in K \bigr\}. 
\]
The polar is closed and convex and contains the origin (see, e.g., Barvinok~\cite{Bar02}). If $K$ is a convex body with $0 \in \interior K$, then $K^\polar$ is a convex body with $0 \in \interior K^\polar$ and $(K^\polar)^\polar = K$. Polarity reverses inclusion, that is, if $G \subseteq K$, then $K^\polar \subseteq G^\polar$. 

We next recall definitions and properties of the Funk and Hilbert metrics. Let $K$ be a convex body in $\RE^d$ with $0 \in \interior K$. For $x_1, \dots, x_m \in \RE^d$, let $\conv\{x_1, \dots, x_m\}$ denote their convex hull. For $y \in \RE^d$, define the affine functional 
\[
        f_y(x) ~ := ~ 1 - \inner{y}{x}, \quad\text{for $x \in K$}.
\]
If $y \in K^\polar$, then $f_y$ is nonnegative on $K$ and strictly positive on $\interior K$. Observe that $f_y(x) = f_x(y)$. For $x_1, x_2 \in \interior K$, the \emph{Funk weak metric} is
\[
    \distFunk[K](x_1, x_2)
        ~ := ~ \sup_{y\in K^\polar} \log\frac{f_y(x_1)}{f_y(x_2)},
\]
and the \emph{Hilbert metric} is its symmetrization: 
\[
    \distHilb[K](x_1, x_2)
        ~ := ~ \frac{1}{2} \bigl( \distFunk[K](x_1, x_2) + \distFunk[K](x_2, x_1) \bigr)
        ~ = ~ \frac{1}{2} \sup_{y,z\in K^\polar}
            \log\frac{f_y(x_1)}{f_y(x_2)}
                      \frac{f_z(x_2)}{f_z(x_1)}. 
\]
It is well known that these definitions are equivalent to the standard definitions, which are based on the chord of $K$ through $x_1$ and $x_2$ (see, e.g., \cite[Corollary 2.6]{PaT14}). In particular, the maximizers $y$ and $z$ in $K^\polar$ correspond to supporting hyperplanes of $K$ at the endpoints of this chord. We refer the reader to \cite{PaT14} for further information.

%-----------------------------------------------------------------------
\subsection{Hilbert Packing Numbers} \label{sec:packing-numbers}
%-----------------------------------------------------------------------

Let $G$ be a nonempty compact subset of $\interior K$, and let $\alpha>0$ (see Figure~\ref{fig:packing}(a)). The (ordinary) \emph{Hilbert packing number}, $\Mord_H(G, K; \alpha)$, is the maximum integer $N$ for which there exist points $x_1,\dots,x_N \in G$ satisfying
\[
    \distHilb[K](x_i,x_j) ~ \geq ~ \alpha, \qquad \text{for $1 \leq i < j \leq N$}
\]
(see Figure~\ref{fig:packing}(b)). For a point $x \in \interior K$ and a nonempty set $C \subset \interior K$, define
\[
    \distHilb[K](x,C)
        ~ := ~ \distHilb[K](C,x)
        ~ := ~ \inf_{c \in C} \distHilb[K](x,c)
        ~ =  ~ \inf_{c \in C} \distHilb[K](c,x).
\]
The \emph{Hilbert convexified packing number}, $\Mconv_H(G,K;\alpha)$, is the maximum integer $N$ for which there exist points $x_1, \dots, x_N \in G$ satisfying
\[
    \distHilb[K]\bigl( x_j, \conv\{x_1,\dots,x_{j-1}\} \bigr)
        ~ \geq ~ \alpha, \qquad \text{for $2 \leq j \leq N$}.
\]
(see Figure~\ref{fig:packing}(c)). We call such a sequence a \emph{Hilbert-convexified $\alpha$-separated sequence}.

%-----------------------------------------------------------------------
\begin{figure}[htbp]
  \centerline{\includegraphics[scale=0.40]{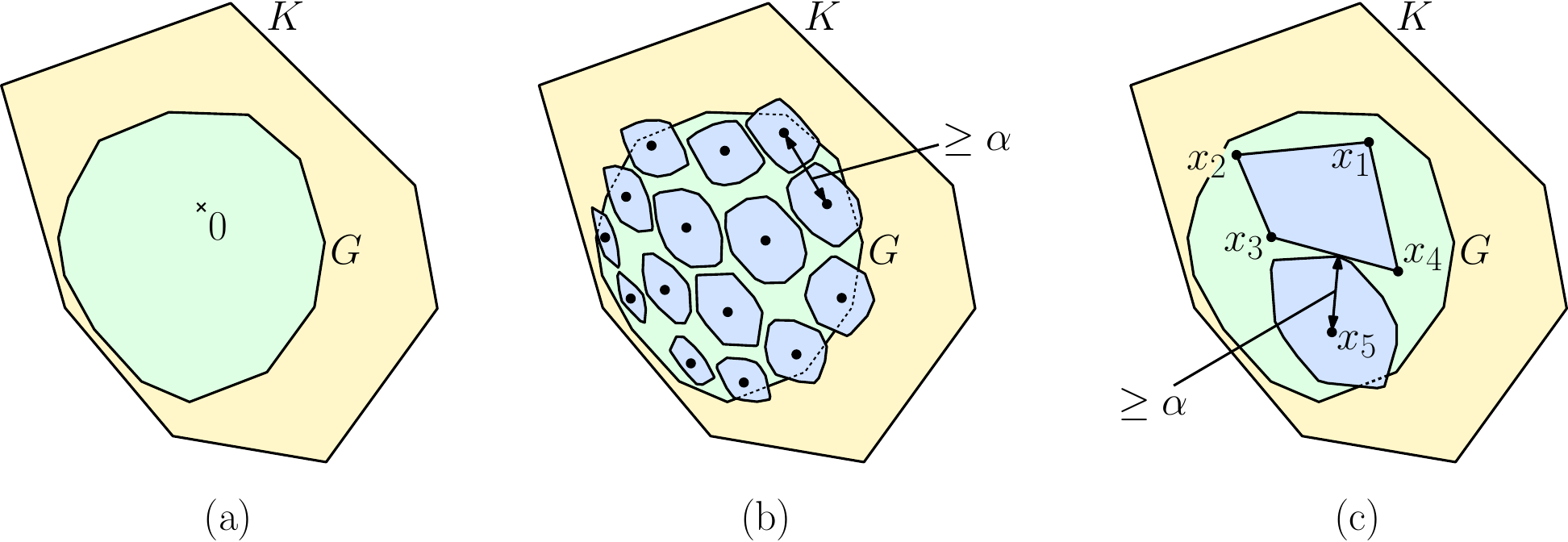}}
  \caption{(a) Convex bodies $G$ and $K$, (b) a Hilbert packing for $\alpha = 0.4$, showing Hilbert balls of radius $0.2$ about each point $x_i$, and (c) a Hilbert convexified $\alpha$-separated sequence for $\alpha = 0.4$, showing $\conv\{x_1, \dots, x_4\}$ and the Hilbert ball of radius $0.4$ about $x_5$.}
  \label{fig:packing}
\end{figure}
%-----------------------------------------------------------------------

Clearly, $\Mconv_H(G, K; \alpha) \leq \Mord_H(G, K; \alpha)$, and both packing numbers are nonincreasing in $\alpha$ and nondecreasing under enlargement of $G$. They are both finite for every $\alpha > 0$ because $G$ is a compact subset of $\interior K$.

%-----------------------------------------------------------------------
\subsection{Polar Potentials} \label{sec:polar-potentials} 
%-----------------------------------------------------------------------

It will be convenient to express the Funk and Hilbert metrics in terms of logarithmic potentials. For $y\in K^\polar$, define the \emph{potential} 
\[
    \phi_y^K(x) ~ := ~ -\log f_y(x), \qquad\text{for $x \in \interior K$}.
\]
For $y,z \in K^\polar$, define the \emph{pair-potential}
\[
    \Phi_{y,z}^K(x)
        ~ := ~ \phi_y^K(x) - \phi_z^K(x)
        ~ = ~ \log \frac{f_z(x)}{f_y(x)}, \qquad\text{for $x \in \interior K$}.
\]
Observe that the pair potential is skew symmetric in the sense that $\Phi_{y,z}^K(x) = -\Phi_{z,y}^K(x)$.

Our assumption that $0 \in \interior G \subset \interior K$ implies that $0 \in \interior K^\polar \subseteq \interior G^\polar$, and combined with the fact that $f_y(x) = f_x(y)$, for $x \in G$ and $y \in K^\polar$, we have
\[
    \phi_z^K(x)
        ~ = ~ \phi_x^{G^\polar}(z).
\]
Consequently, for $x_1,x_2\in G$, the Funk distance has the equivalent representations
\[
    \distFunk[K](x_1,x_2)
        ~ = ~ \sup_{y\in K^\polar}
            \bigl(\phi_y^K(x_2)-\phi_y^K(x_1)\bigr)
        ~ = ~ \sup_{y\in K^\polar}
            \Phi_{x_2,x_1}^{G^\polar}(y).
\]
Symmetrizing the Funk distance gives two useful representations of the Hilbert distance. For $x_1,x_2\in G$,
\begin{align}
    \distHilb[K](x_1,x_2)
        & ~ = ~ \frac{1}{2} \sup_{y,z\in K^\polar} \bigl( \Phi_{y,z}^K(x_2)-\Phi_{y,z}^K(x_1) \bigr), \label{eq:hilbert-pair-potential} \\
    \distHilb[K](x_1,x_2)
        & ~ = ~ \frac{1}{2} \left( \sup_{y\in K^\polar} \Phi_{x_1,x_2}^{G^\polar}(y) - \inf_{y\in K^\polar}
\Phi_{x_1,x_2}^{G^\polar}(y) \right). \label{eq:hilbert-oscillation}
\end{align}

The maps $(y,x)\longmapsto\phi_y^K(x)$ and $(y,z,x)\longmapsto\Phi_{y,z}^K(x)$ are continuous on $K^\polar\times\interior K$ and $K^\polar\times K^\polar\times\interior K$, respectively. In particular, both maps are uniformly continuous when the $x$-variable is restricted to a compact subset of $\interior K$.

%=======================================================================
\section{Two-pole Separation and Polar Oscillation} \label{sec:two-pole-tools}
%=======================================================================

Recall that $K$ is a convex body in $\RE^d$ with $0 \in \interior K$. Since the Hilbert metric is the symmetrization of the Funk metric, its natural separation certificate consists of two polar witnesses. We use normalized nonnegative affine functions and a minimax argument to choose such a pair uniformly on a convex hull. We then bound the oscillation of the resulting pair-potentials by convexified packing in the polar Hilbert geometry.

%-----------------------------------------------------------------------
\subsection{Normalized Affine Functions} \label{sec:normalized-affine-functions}
%-----------------------------------------------------------------------

Let $\Aff(K)$ denote the finite-dimensional vector space of restrictions to $K$ of real-valued affine functions on $\RE^d$. For $x \in \interior K$, define
\[
    \mathcal{A}_x^K
        ~ := ~ \{ a \in \Aff(K) \ST a \geq 0 \text{ on } K,\ a(x) = 1 \}.
\]
For $y \in K^\polar$, define the normalized affine function
\[
    a_y^x(u)
        ~ := ~ \frac{f_y(u)}{f_y(x)}, \qquad\text{for $u \in K$}.
\]
The denominator is positive because $x \in \interior K$. The following lemma shows that all elements of $\mathcal{A}_x^K$ are of this form.

%-----------------------------------------------------------------------
\begin{lemma} \label{lem:normalized-affine-functions}
For every $x \in \interior K$, $\mathcal{A}_x^K = \{ a_y^x \ST y \in K^\polar \}$. In particular, $\mathcal{A}_x^K$ is compact and convex.
\end{lemma}
%-----------------------------------------------------------------------

%-----------------------------------------------------------------------
\begin{proof}
If $y \in K^\polar$, then $f_y \geq 0$ on $K$ and $f_y(x) > 0$. Thus $a_y^x \in \mathcal{A}_x^K$. Conversely, given $a \in \mathcal{A}_x^K$, we can express it equivalently as
\[
    a(u)
        ~ = ~ a(0) + \inner{q}{u}.
\]
We first observe that $a(0) > 0$. To see this, the fact that $0 \in \interior K$ implies that $a$ is nonnegative on a neighborhood of $0$. If, to the contrary, $a(0)$ were $0$, then $0$ is a local minimum of the affine function $a$, and hence $q = 0$. This would imply $a \equiv 0$, contradicting $a(x) = 1$.

Set $y := -q/a(0)$. We will show that $y \in K^\polar$ and $a_y^x = a$. For all $u \in K$, we have
\[
    f_y(u)
        ~ = ~ 1 + \frac{\inner{q}{u}}{a(0)}
        ~ = ~ \frac{a(u)}{a(0)}.
\]
Since $a \geq 0$ on $K$, it follows that $f_y \geq 0$ on $K$, and hence $y \in K^\polar$. Moreover, $f_y(x) = a(x)/a(0) = 1/a(0)$, so
\[
    a_y^x(u)
        ~ = ~ \frac{f_y(u)}{f_y(x)}
        ~ = ~ a(u),
\]
as desired. 

Observe that the map $y\longmapsto a_y^x$ from $K^\polar$ to $\Aff(K)$ is continuous. Since $K^\polar$ is compact and its image is $\mathcal{A}_x^K$, the latter set is also compact. Convexity follows directly from its definition. 
\end{proof}
%-----------------------------------------------------------------------

Notice also that every $a \in \mathcal{A}_x^K$ is strictly positive on $\interior K$. Indeed, a nonnegative affine function that vanishes at an interior point must vanish identically, contradicting $a(x) = 1$.

%-----------------------------------------------------------------------
\subsection{Two-pole Separation} \label{sec:two-pole-separation}
%-----------------------------------------------------------------------

Pointwise Hilbert separation provides a pair of affine witnesses for each point of the preceding convex hull, but the pair may depend on that point. The minimax argument below produces a single pair that works uniformly on the entire hull. 

%-----------------------------------------------------------------------
\begin{lemma} \label{lem:finite-two-pole-separation}
Let $x_1, \dots, x_j \in \interior K$, where $j \geq 2$, and let $C = \conv\{x_1,\dots,x_{j-1}\}$. If $\distHilb[K](C,x_j) \geq \alpha$, then there exist $y_j,z_j \in K^\polar$ such that, with
\[
    \Phi_j ~ := ~ \Phi_{y_j,z_j}^K,
\]
one has $\Phi_j(x_j) - \Phi_j(c) \geq 2 \alpha$, for all $c \in C$. In particular,
\[
    \Phi_j(x_j) - \Phi_j(x_i)
        ~ \geq ~ 2 \alpha \qquad \text{for $1 \leq i < j$}.
\]
\end{lemma}
%-----------------------------------------------------------------------

%-----------------------------------------------------------------------
\begin{proof}
To simplify notation, let $x := x_j$. For every $c \in C$,
\[
    2 \cdot \distHilb[K](c,x)
        ~ = ~ \max_{y,z \in K^\polar} \log\left( \frac{f_y(c)}{f_y(x)} \frac{f_z(x)}{f_z(c)} \right) 
        ~ = ~ \max_{y,z \in K^\polar} \log \frac{a_y^{x}(c)}{a_z^{x}(c)}.
\]
As $y$ and $z$ range independently over $K^\polar$, Lemma~\ref{lem:normalized-affine-functions} yields
\[
    2  \cdot \distHilb[K](c,x)
        ~ = ~ \max_{p,q \in \mathcal{A}_x^K} \log\frac{p(c)}{q(c)}.
\]

Next, for $(p,q) \in \mathcal{A}_x^K \times \mathcal{A}_x^K$ and $c\in C$, define $L_{\alpha}(p, q; c) := p(c) - e^{2\alpha} q(c)$. Since $\distHilb[K](c,x) \geq \alpha$, the preceding identity implies that $\max_{p,q \in \mathcal{A}_x^K} L_{\alpha}(p, q; c)$ is nonnegative for every $c \in C$. The sets $\mathcal{A}_x^K \times \mathcal{A}_x^K$ and $C$ are compact and convex, and $L_{\alpha}$ is continuous and affine in each variable. Therefore, by Sion's Minimax Theorem \cite{Sio58}, we have
\[
    \max_{p,q \in \mathcal{A}_x^K} \min_{c \in C} L_{\alpha}(p, q; c)
        ~ = ~ \min_{c \in C} \max_{p,q \in \mathcal{A}_x^K} L_{\alpha}(p, q; c)
        ~ \geq ~ 0.
\]
Consequently, there exist $p,q\in\mathcal{A}_x^K$ such that
\[
    p(c) ~ \geq ~ e^{2\alpha} q(c), \qquad\text{for all $c \in C$}.
\]
By Lemma~\ref{lem:normalized-affine-functions}, there exist $y_j, z_j \in K^\polar$ such that $p = a_{y_j}^x$ and $q = a_{z_j}^x$. Taking logarithms implies that for all $c\in C$,
\begin{align*}
    2 \alpha
        & ~ \leq ~ \log\frac{a_{y_j}^x(c)}{a_{z_j}^x(c)} 
          ~ = ~ \log \left( \frac{f_{y_j}(c)}{f_{y_j}(x)} \frac{f_{z_j}(x)}{f_{z_j}(c)} \right)  \\
        & ~ = ~ \Phi_{y_j,z_j}^K(x) - \Phi_{y_j,z_j}^K(c).
\end{align*}
Recalling that $x = x_j$ proves the first assertion. Taking $c = x_i$ yields the second.
\end{proof}
%-----------------------------------------------------------------------

%-----------------------------------------------------------------------
\subsection{Polar Oscillation} \label{sec:polar-oscillation}
%-----------------------------------------------------------------------

The scale-selection argument in the next section depends on the range of the pair-potentials on $G$. Define the \emph{total oscillation} by 
\[
    \Omega_2(G, K)
        ~ := ~ \sup_{\substack{x \in G \\ y,z \in K^\polar}} \Phi_{y,z}^K(x) - \inf_{\substack{x \in G \\ y,z \in K^\polar}} \Phi_{y,z}^K(x).
\]
The extrema in this definition are attained by continuity and compactness. The following lemma identifies a chord of $K^\polar$ whose Hilbert length is at least $\Omega_2(G, K)/4$.

%-----------------------------------------------------------------------
\begin{lemma} \label{lem:omega-realization}
There exist $y,z \in K^\polar$ such that
\[
    \Omega_2(G, K)
        ~ \leq ~ 4 \cdot \distHilb[G^\polar](y,z).
\]
\end{lemma}

\begin{proof}
The skew symmetry of $\Phi^K_{y,z}$ implies that the set of values attained by the pair-potentials is symmetric about zero. Thus, by compactness, there exist $x \in G$ and $y,z \in K^\polar$ such that
\[
    \Phi^K_{y,z}(x) 
        ~ = ~ \frac{\Omega_2(G,K)}{2}.
\]
Applying Eq.~\eqref{eq:hilbert-oscillation} to $y,z$ gives
\begin{align*}
    2 \cdot \distHilb[G^\polar](y,z)
        & ~ = ~ \sup_{u \in G} \Phi^K_{y,z}(u) - \inf_{u \in G} \Phi^K_{y,z}(u) \\
        & ~ \geq ~ \Phi^K_{y,z}(x) - \Phi^K_{y,z}(0)
          ~ = ~ \Phi^K_{y,z}(x)
          ~ = ~ \frac{\Omega_2(G,K)}{2},
\end{align*}
which implies the desired bound.
\end{proof}
%-----------------------------------------------------------------------

The following lemma bounds this oscillation by the convexified packing number in the polar Hilbert geometry.

%-----------------------------------------------------------------------
\begin{lemma}[Polar oscillation estimate] \label{lem:omega-convexified}
For every $\alpha > 0$,
\[
    1 + \frac{\Omega_2(G,K)}{\alpha}
        ~ \leq ~ 5 \, \Mconv_H(K^\polar, G^\polar; \alpha).
\]
\end{lemma}
%-----------------------------------------------------------------------

%-----------------------------------------------------------------------
\begin{proof}
Lemma~\ref{lem:omega-realization} implies that there exist $y,z \in K^\polar$ such that
\[
    \distHilb[G^\polar](y,z)
        ~ \geq ~ \frac{\Omega_2(G,K)}{4}.
\]
Let $D$ denote this Hilbert distance. It is well known that the restriction of the Hilbert metric to an open chord is additive along ordered points \cite{PaT14}. That is, if $p,q,r$ occur in this order on a chord of $G^\polar$, then
\[
    \distHilb[G^\polar](p,r)
        ~ = ~ \distHilb[G^\polar](p,q) + \distHilb[G^\polar](q,r).
\]
Set $m := \floor{D/\alpha} + 1$. Using additivity along the chord $[y,z]$, select a sequence of points
\[
        v_1, \dots, v_m \in [y,z]
\]
in their natural order such that $v_1 = y$ and $\distHilb[G^\polar](v_i, v_{i+1}) = \alpha$, for $1 \leq i < m$. (When $m = 1$, only the point $v_1 = y$ is needed.) For $2 \leq j \leq m$, the points are collinear and ordered, so
\[
    \conv\{v_1, \dots, v_{j-1}\} ~ = ~ [v_1, v_{j-1}]
    \qquad\text{and}\qquad
    \distHilb[G^\polar] \bigl( \conv\{v_1,\dots,v_{j-1}\}, v_j \bigr) ~ = ~ \alpha.
\]
 Thus $v_1, \dots ,v_m$ forms a Hilbert-convexified $\alpha$-separated sequence in $K^\polar$. Therefore, considering just this chord, we have
\[
    \Mconv_H(K^\polar, G^\polar; \alpha)
        ~ \geq ~ m
        ~ =    ~ \floor{\frac{D}{\alpha}} + 1
        ~ \geq ~ \frac{D}{\alpha}
        ~ \geq ~ \frac{\Omega_2(G,K)}{4\alpha}.
\]
Since $K^\polar$ is not empty, the same packing number is at least $1$. Combining the two gives
\[
    1 + \frac{\Omega_2(G,K)}{\alpha}
        ~ \leq ~ \Mconv_H(K^\polar, G^\polar; \alpha) + 4 \, \Mconv_H(K^\polar,G^\polar; \alpha)
        ~ =    ~ 5 \, \Mconv_H(K^\polar, G^\polar; \alpha),
\]
as desired. 
\end{proof}
%-----------------------------------------------------------------------

%=======================================================================
\section{Diagonal Convexification and the Mixed Primal--polar Bound} \label{sec:diagonal-convexification}
%=======================================================================

We now combine the finite-hull two-pole separation theorem with a scale decomposition of the resulting polar witnesses. Within each scale class, the witness pairs form a diagonally convexified packing in $K^\polar \times K^\polar$. A mixed product estimate for this packing then yields the main result. 

%-----------------------------------------------------------------------
\subsection{Diagonal Convexified Packing} \label{sec:diagonal-convexified-packing}
%-----------------------------------------------------------------------

Each primal point produces two polar witnesses. When earlier witness pairs are convexified, the two coordinates must be averaged with the same coefficients, because they arise from a single convex combination of the preceding primal points. This leads to the following notion. 

%-----------------------------------------------------------------------
\begin{definition}[Diagonal convexified packing] \label{def:diagonal-packing}
For $\eta > 0$, let
\[
    \Mconv_{\mathrm{diag},H} (K^\polar, G^\polar; \eta)
\]
be the supremum of all $N \in \NN$ for which there exists a sequence $\ang{(y_1, z_1), \dots, (y_N,z_N)}$ of points from $K^\polar \times K^\polar$ such that, for every $2\leq j\leq N$ and every choice of coefficients $\lambda_i \geq 0$ and $\sum_{i<j} \lambda_i = 1$, one has
\[
    \distHilb[G^\polar] \!\left( \left( \sum\nolimits_{i<j} \lambda_i y_i \right), y_j \right) + \distHilb[G^\polar] \!\left( \left( \sum\nolimits_{i<j} \lambda_i z_i \right), z_j \right)
        ~ \geq ~ \eta.
\]
If such integers $N$ are unbounded, we set $\Mconv_{\mathrm{diag},H} (K^\polar,G^\polar; \eta) = +\infty$.
\end{definition}
%-----------------------------------------------------------------------

Observe that this can be viewed as a type of convexified packing over $K^\polar \times K^\polar$. The name derives from the fact that because the same $\lambda_i$ values are used in both convex combinations, points are convexified in pairs:
\[
    \conv \{ (y_i,z_i) \ST i<j \}
        ~ \subset ~ K^\polar\times K^\polar,
\]
rather than independently. This synchronization of the coefficients is the essential feature of the definition.

The following elementary estimate shows that if two witness pairs give substantially different pair-potential values at the same primal point, then the sum of the Hilbert distances between their corresponding polar coordinates is large.

%-----------------------------------------------------------------------
\begin{lemma} \label{lem:synchronized-ratio-estimate}
For $y, z, y',z'\in K^\polar$ and $x\in G$,
\[
    \distHilb[G^\polar](y,y') + \distHilb[G^\polar](z,z')
        ~ \geq ~ \frac{1}{2} \abs{\Phi_{y,z}^K(x) - \Phi_{y',z'}^K(x)}.
\]
\end{lemma}
%-----------------------------------------------------------------------

%-----------------------------------------------------------------------
\begin{proof}
We first claim that, for every $u,u'\in K^\polar$,
\[
    2\distHilb[G^\polar](u,u')
        ~ \geq ~
        \abs{\phi_u^K(x)-\phi_{u'}^K(x)}.
\]
Indeed, applying Eq.~\eqref{eq:hilbert-pair-potential} with $G^\polar$ as the ambient body, and using the pair $(x,0) \in G \times G$, gives
\[
    2\distHilb[G^\polar](u,u')
        ~ \geq ~ \Phi_{x,0}^{G^\polar}(u') - \Phi_{x,0}^{G^\polar}(u)
        ~ = ~ \phi_x^{G^\polar}(u') - \phi_x^{G^\polar}(u)
        ~ = ~ \phi_{u'}^K(x) - \phi_u^K(x).
\]
Interchanging $u$ and $u'$ proves the claim.

Applying the claim to $(u,u') = (y,y')$ and $(u,u') = (z,z')$ gives
\begin{align*}
    2\bigl( \distHilb[G^\polar](y,y') + \distHilb[G^\polar](z,z') \bigr)
        & ~ \geq ~ \abs{\phi_y^K(x) - \phi_{y'}^K(x)} + \abs{\phi_z^K(x) - \phi_{z'}^K(x)} \\
        & ~ \geq ~ \abs{\bigl(\phi_y^K(x)-\phi_z^K(x)\bigr) - \bigl(\phi_{y'}^K(x)-\phi_{z'}^K(x)\bigr)} \\
        & ~ =    ~ \abs{\Phi_{y,z}^K(x) - \Phi_{y',z'}^K(x)}.
\end{align*}
Dividing by $2$ proves the assertion.
\end{proof}
%-----------------------------------------------------------------------

%-----------------------------------------------------------------------
\subsection{Scale Decomposition and Diagonal Reduction} \label{sec:scale-selected-witness-pairs}
%-----------------------------------------------------------------------

The witness pair associated with a primal point has a relative normalization, measured by the value of its pair-potential at that point. The polar oscillation estimate allows us to partition the witnesses into a controlled number of classes in which these normalizations are comparable. 

%-----------------------------------------------------------------------
\begin{lemma} \label{lem:scale-selected-diagonal-reduction}
For every $\alpha > 0$,
\[
    \Mconv_H(G, K; \alpha)
        ~ \leq ~ 1 + 10 \, \Mconv_H(K^\polar, G^\polar; \alpha) \, \Mconv_{\mathrm{diag},H} \!\left( K^\polar, G^\polar; \frac{3\alpha}{4} \right).
\]
\end{lemma}
%-----------------------------------------------------------------------

%-----------------------------------------------------------------------
\begin{proof}
Let $x_1,\dots,x_N\in G$ be Hilbert-convexified $\alpha$-separated. The assertion is immediate when $N = 1$, so assume that $N \geq 2$. For each $2 \leq j \leq N$, apply Lemma~\ref{lem:finite-two-pole-separation} to choose $y_j, z_j \in K^\polar$ such that for all $i < j$,
\[
    \Phi_{y_j,z_j}^K(x_j) - \Phi_{y_j,z_j}^K(x_i)
        ~ \geq ~ 2 \alpha.
\]
Set $s_j := \Phi_{y_j,z_j}^K(x_j)$. Since $\Phi_{y_j,z_j}^K(x_i) = \log\frac{f_{z_j}(x_i)}{f_{y_j}(x_i)}$,
the above inequality is equivalent to
\begin{equation}
    f_{y_j}(x_i)
        ~ \geq ~ e^{2\alpha - s_j}f_{z_j}(x_i),
        \qquad\text{for $i<j$}. \label{eq:scale-selected-witness-inequality}
\end{equation}
By the definition of $\Omega_2(G,K)$, the numbers $s_2,\dots,s_N$ lie in an interval of length at most $\Omega_2(G,K)$. Partition this interval into subintervals of length at most $\alpha/2$. The number $L$ of subintervals may be chosen so that
\[
    L
        ~ \leq ~ 1 + \frac{2\,\Omega_2(G,K)}{\alpha}
        ~ \leq ~ 2 \left( 1 + \frac{\Omega_2(G,K)}{\alpha} \right)
        ~ \leq ~ 10 \, \Mconv_H(K^\polar, G^\polar; \alpha),
\]
where the last inequality follows from Lemma~\ref{lem:omega-convexified}. 

Fix one nonempty scale class and write its indices in increasing order as $j_1 < \dots < j_m$. Since the corresponding values $s_{j_q}$ lie in an interval of length at most $\alpha/2$, there exists $t \in \RE$ such that
\[
    \abs{s_{j_q}-t}
        ~ \leq ~ \frac{\alpha}{4},
        \qquad\text{for $1 \leq q \leq m$}.
\]
We assert that the reverse-ordered sequence of witness pairs $\ang{(y_{j_m}, z_{j_m}), \dots, (y_{j_1}, z_{j_1})}$ is diagonally convexified $3\alpha/4$-separated. To show this, fix $1 \leq q < m$, and choose arbitrary nonnegative coefficients $\ang{\lambda_{q+1}, \dots, \lambda_m}$ summing to $1$ and define the convex combinations
\[
        \bar y  ~ := ~ \sum\nolimits_{\ell=q+1}^m \lambda_{\ell} y_{j_{\ell}}
        \qquad\text{and}\qquad
        \bar z ~ := ~ \sum\nolimits_{\ell=q+1}^m \lambda_{\ell} z_{j_{\ell}}
\]
of the pairs that precede $(y_{j_q},z_{j_q})$ in the reversed ordering. Because $j_q < j_{\ell}$, Eq.~\eqref{eq:scale-selected-witness-inequality}, applied with $i = j_q$ and $j = j_{\ell}$, gives
\[
    f_{y_{j_{\ell}}}(x_{j_q})
        ~ \geq ~ e^{2\alpha - s_{j_{\ell}}} f_{z_{j_{\ell}}}(x_{j_q})
        ~ \geq ~ e^{2\alpha - t - \frac{\alpha}{4}} f_{z_{j_{\ell}}}(x_{j_q}).
\]
Since $f_y(x)$ is affine in $y$, averaging these inequalities with the coefficients $\lambda_{\ell}$ yields
\[
    f_{\bar y}(x_{j_q})
        ~ \geq ~ e^{2\alpha - t - \frac{\alpha}{4}} f_{\bar z}(x_{j_q}).
\]
On the other hand,
\[
    \frac{f_{y_{j_q}}(x_{j_q})}{f_{z_{j_q}}(x_{j_q})}
        ~ = ~ e^{-s_{j_q}}
        ~ \leq ~ e^{-t + \frac{\alpha}{4}}.
\]
Consequently,
\[
    \Phi_{y_{j_q},z_{j_q}}^K(x_{j_q}) - \Phi_{\bar y,\bar z}^K(x_{j_q})
        ~ = ~ \log\left( \frac{f_{\bar y}(x_{j_q})/f_{\bar z}(x_{j_q})}{f_{y_{j_q}}(x_{j_q})/f_{z_{j_q}}(x_{j_q})}
        \right)
        ~ \geq ~ \frac{3\alpha}{2}.
\]
Lemma~\ref{lem:synchronized-ratio-estimate}, applied at $x_{j_q}$, therefore gives
\[
    \distHilb[G^\polar](\bar y,y_{j_q}) + \distHilb[G^\polar](\bar z,z_{j_q})
        ~ \geq ~ \frac{3\alpha}{4}.
\]
Since the coefficients were arbitrary, this proves the claim. Hence
\[
    m
        ~ \leq ~ \Mconv_{\mathrm{diag},H} \!\left( K^\polar, G^\polar; \frac{3\alpha}{4} \right).
\]
There are at most $L$ nonempty scale classes. Accounting separately for $x_1$, which has no associated witness pair, gives
\[
    N
        ~ \leq ~ 1 + L \, \Mconv_{\mathrm{diag},H} \!\left( K^\polar, G^\polar; \frac{3\alpha}{4} \right)
        ~ \leq ~ 1 + 10 \, \Mconv_H(K^\polar,G^\polar; \alpha) \, \Mconv_{\mathrm{diag},H} \!\left( K^\polar, G^\polar; \frac{3\alpha}{4} \right).
\]
Taking the supremum over all Hilbert-convexified $\alpha$-separated sequences in $G$ proves the result. 
\end{proof}
%-----------------------------------------------------------------------

%-----------------------------------------------------------------------
\subsection{The Diagonal Product Estimate} \label{sec:diagonal-product-estimate}
%-----------------------------------------------------------------------

We next bound the size of diagonal convexified packings. We first partition one coordinate with an ordinary packing net. Then within each part, we show that the convexity of Hilbert balls forces the other coordinate to form a convexified packing.

%-----------------------------------------------------------------------
\begin{lemma}[Mixed diagonal product bound] \label{lem:mixed-diagonal-product-bound}
For every $\eta > 0$,
\[
    \Mconv_{\mathrm{diag},H}(K^\polar, G^\polar; \eta)
        ~ \leq ~ \Mord_H \!\left( K^\polar, G^\polar; \frac{\eta}{4} \right) \Mconv_H \!\left( K^\polar, G^\polar; \frac{\eta}{2} \right).
\]
\end{lemma}
%-----------------------------------------------------------------------

%-----------------------------------------------------------------------
\begin{proof}
Let $\ang{(y_1,z_1),\dots,(y_N,z_N)}$ be a diagonally convexified $\eta$-separated sequence from $K^\polar \times K^\polar$, and set $r := \eta/4$. Select a maximal Hilbert $r$-separated subset $\mathcal{N} \subset K^\polar$. Clearly,
\[
    |\mathcal{N}|
        ~ \leq ~ \Mord_H(K^\polar, G^\polar; r).
\]
Moreover, maximality implies that $\mathcal{N}$ is an open $r$-net for $K^\polar$, that is, for every $y \in K^\polar$, there exists $u \in \mathcal{N}$ such that $\distHilb[G^\polar](y,u) < r$.

For each $j$, choose $u(j)\in\mathcal{N}$ such that $\distHilb[G^\polar](y_j,u(j)) < r$ and partition the indices according to the value of $u(j)$. Fix one nonempty block whose indices in their original order are $j_1 < \dots < j_m$, and let $u \in \mathcal{N}$ be their common net point. We claim that $\ang{z_{j_1}, \dots, z_{j_m}}$ is a convexified $\eta/2$-separated sequence in the Hilbert geometry of $G^\polar$. To see this, fix $2 \leq q \leq m$, arbitrary nonnegative coefficients $\ang{\lambda_1, \dots, \lambda_{q-1}}$ summing to $1$, and define the convex combinations
\[
        \bar y ~ := ~ \sum\nolimits_{\ell=1}^{q-1} \lambda_{\ell} y_{j_{\ell}}
        \qquad\text{and}\qquad
        \bar z ~ := ~ \sum\nolimits_{\ell=1}^{q-1} \lambda_{\ell} z_{j_{\ell}}.
\]

It is well known that open balls in any Hilbert geometry are convex (see, e.g., \cite[Section~18]{Bus55}). Since every $y_{j_{\ell}}$ belongs to the open ball of radius $r$ centered at $u$, so does $\bar y$. Thus $\distHilb[G^\polar](\bar y,u) < r$. Together with the fact that  $\distHilb[G^\polar](y_{j_q},u) < r$, the triangle inequality gives
\[
    \distHilb[G^\polar](\bar y, y_{j_q})
        ~ < ~ 2 r
        ~ = ~ \frac{\eta}{2}.
\]
Diagonal separation now implies
\[
    \distHilb[G^\polar](\bar z, z_{j_q})
        ~ > ~ \frac{\eta}{2}.
\]
Since this holds for every convex combination $\bar z$ of $z_{j_1},\dots,z_{j_{q-1}}$, it follows that
\[
    \distHilb[G^\polar]\bigl(\conv\{z_{j_1}, \dots, z_{j_{q-1}}\},z_{j_q}\bigr)
        ~ \geq ~ \frac{\eta}{2}.
\]
The non-strict inequality follows upon taking the infimum over the preceding convex hull. Therefore, the $z_{j_{\ell}}$'s form a convexified $\eta/2$-separated sequence, as desired, implying that
\[
    m
        ~ \leq ~ \Mconv_H \!\left( K^\polar, G^\polar; \frac{\eta}{2} \right).
\]
As there are at most $|\mathcal{N}| \leq \Mord_H(K^\polar, G^\polar; \eta/4)$ blocks, we have
\[
    N
        ~ \leq ~ \Mord_H \!\left( K^\polar, G^\polar; \frac{\eta}{4} \right) \Mconv_H \!\left( K^\polar, G^\polar; \frac{\eta}{2} \right).
\]
Taking the supremum over all diagonally convexified $\eta$-separated sequences proves the lemma.
\end{proof}
%-----------------------------------------------------------------------

The following consequence is immediate from monotonicity in the separation parameter and the inequality $\Mconv_H \leq \Mord_H$.

%-----------------------------------------------------------------------
\begin{corollary} \label{cor:ordinary-diagonal-product-bound}
For every $\eta > 0$,
\[
    \Mconv_{\mathrm{diag},H}(K^\polar, G^\polar; \eta)
        ~ \leq ~ \Mord_H \!\left( K^\polar, G^\polar; \frac{\eta}{4} \right)^{\! 2}.
\]
\end{corollary}
%-----------------------------------------------------------------------

%-----------------------------------------------------------------------
\subsection{The Mixed Primal--polar Bound} \label{sec:mixed-primal-polar-bound}
%-----------------------------------------------------------------------

By combining the scale-selected reduction with the diagonal product estimate, we can now prove the main theorem with the constants explicitly given. 

%-----------------------------------------------------------------------
\begin{lemma}
For every pair of convex bodies $G$ and $K$ in $\RE^d$ satisfying $0 \in \interior G$ and $G \subset \interior K$ and every $\alpha > 0$,
\[
    \Mconv_H(G, K; \alpha)
        ~ \leq ~ 11 \, \Mconv_H \!\left( K^\polar, G^\polar; \frac{3\alpha}{16} \right)^{\! 2} \Mord_H \!\left( K^\polar, G^\polar; \frac{3\alpha}{16} \right).
\]
\end{lemma}
%-----------------------------------------------------------------------

%-----------------------------------------------------------------------
\begin{proof}
Applying Lemmas~\ref{lem:scale-selected-diagonal-reduction} and~\ref{lem:mixed-diagonal-product-bound} with $\eta := 3\alpha/4$, we obtain
\begin{align*}
    \Mconv_H(G, K; \alpha)
        & ~ \leq ~ 1 + 10 \, \Mconv_H(K^\polar, G^\polar; \alpha) \, \Mconv_{\mathrm{diag},H} \!\left( K^\polar, G^\polar; \frac{3\alpha}{4} \right) \\
        & ~ \leq ~ 1 + 10 \, \Mconv_H(K^\polar, G^\polar; \alpha) \, \Mconv_H \!\left( K^\polar, G^\polar; \frac{3\alpha}{16} \right) \Mord_H \!\left( K^\polar, G^\polar; \frac{3\alpha}{16} \right),
\end{align*}
where, in applying Lemma~\ref{lem:mixed-diagonal-product-bound}, we have used the monotonicity of separation numbers with respect to the distance parameter. 

Set $c := 3/16$. Since $c \alpha \leq \alpha$, monotonicity implies that
\[
    \Mconv_H(K^\polar, G^\polar; \alpha)
        ~ \leq ~ \Mconv_H(K^\polar, G^\polar; c\alpha).
\]
Therefore
\[
    \Mconv_H(G, K; \alpha)
        ~ \leq ~ 1 + 10 \, \Mconv_H(K^\polar, G^\polar; c\alpha)^2 \Mord_H(K^\polar, G^\polar; c\alpha).
\]
Each packing number in the product is at least $1$, so the additive term is absorbed by increasing the constant from $10$ to $11$. This proves the theorem with $C = 11$. 
\end{proof}
%-----------------------------------------------------------------------

This proves Theorem~\ref{thm:main} for the constants $C = 11$ and $c = 3/16$. Since convexified separation numbers cannot exceed ordinary separation numbers, we obtain the following corollary, which avoids the mixed product.

%-----------------------------------------------------------------------
\begin{corollary} \label{cor:ordinary-cubic-bound}
For every pair of convex bodies $G$ and $K$ in $\RE^d$ satisfying $0 \in \interior G$ and $G \subset \interior K$ and every $\alpha > 0$,
\[
    \Mconv_H(G, K; \alpha)
        ~ \leq ~ 11 \, \Mord_H \!\left( K^\polar, G^\polar; \frac{3\alpha}{16} \right)^{\! 3}.
\]
\end{corollary}
%-----------------------------------------------------------------------

%=======================================================================
\section{Concluding Remarks and Open Problems} \label{sec:conclusion}
%=======================================================================

Our main theorem gives a dimension-free primal--polar estimate for Hilbert convexified packing. Its mixed form reflects the two components of the proof. The scale decomposition contributes one convexified polar packing factor, while the diagonal product estimate contributes one convexified and one ordinary polar packing factor. Thus, the ordinary packing term enters only at the final step, where one coordinate of the polar witness pairs is partitioned using an ordinary packing net. 

The mixed form is nevertheless sufficient for convexified packing to serve as a bridge to ordinary packing duality. Indeed, suppose that there exist absolute constants $C_0, c_0 > 0$ and $q \geq 1$ such that
\begin{equation}
    \Mord_H(G, K; \alpha)
        ~ \leq ~ C_0 \, \Mconv_H(G, K; c_0 \alpha)^q, \label{eq:conclusion-ordinary-by-convexified}
\end{equation}
for every admissible pair $G \subset\interior K$. Applying Theorem~\ref{thm:main} to the right-hand side of Eq.~\eqref{eq:conclusion-ordinary-by-convexified}, and then using $\Mconv_H \leq \Mord_H$, would yield
\[
    \Mord_H(G, K; \alpha)
        ~ \leq ~ C_1\, \Mord_H(K^\polar, G^\polar; c_1\alpha)^{3 q},
\]
for suitable absolute constants $C_1, c_1 > 0$. Consequently, the presence of one ordinary polar packing factor does not obstruct its role as the duality bridge. The remaining task would be to establish the intrinsic estimate~\eqref{eq:conclusion-ordinary-by-convexified}, which concerns only one Hilbert geometry at a time and constitutes the second, geometric, stage of the program. 

A separate question is whether the mixed diagonal product estimate can itself be made fully convexified.

%-----------------------------------------------------------------------
\begin{problem}[Fully convexified diagonal product bound] \label{prob:fully-convexified-diagonal-bound}
Do there exist absolute constants $C,c>0$ such that, for every admissible pair $G \subset \interior K$ and every $\eta > 0$, the following holds?
\[
    \Mconv_{\mathrm{diag},H}(K^\polar, G^\polar; \eta)
        ~ \leq ~ C\, \Mconv_H(K^\polar, G^\polar; c\eta)^2.
\]
\end{problem}
%-----------------------------------------------------------------------

An affirmative answer would replace the ordinary packing factor in Theorem~\ref{thm:main} by a third convexified factor and would give
\[
    \Mconv_H(G,K; \alpha)
        ~ \leq ~ C\, \Mconv_H(K^\polar, G^\polar; c\alpha)^3.
\]
The difficulty is that diagonal convexification requires the same coefficients in the two polar coordinates. The proof of Lemma~\ref{lem:mixed-diagonal-product-bound} avoids this synchronization problem by bounding one coordinate through an ordinary-packing net. It remains to determine whether this use of ordinary packing is essential or merely a feature of the present argument. 

%=======================================================================
\section*{Acknowledgments}
%=======================================================================

The authors acknowledge the use of generative AI tools in the preparation of this manuscript. The authors take full responsibility for its content.

%=======================================================================
% Bibliography
%=======================================================================

\end{document}